\documentclass[10pt]{elsarticle}
\usepackage[utf8]{inputenc}
\usepackage{amsmath}
\usepackage{amssymb}
\usepackage{amsfonts}
\usepackage{amsthm}

\journal{Arxiv}

\newtheorem{lemma}{Lemma}
\newtheorem{corollary}{Corollary}
\newtheorem{problem}{Problem}
\newtheorem{proposition}{Proposition}

\begin{document}

\begin{frontmatter}
\title{On balanced characteristic functions of canonical cliques in Paley graphs of square order}

\author[01]{Sergey Goryainov}
\ead{sergey.goryainov3@gmail.com}

\author[02]{Huiqiu Lin}
\ead{huiqiulin@126.com}

\address[01] {Department of Mathematics, Chelyabinsk State University, Brat'ev Kashirinyh st. 129, Chelyabinsk  454021, Russia}
\address[02]{Department of Mathematics, East China University of Science and Technology,
Shanghai 200237, P.R. China}

\begin{abstract}
In this paper we prove that balanced characteristic functions of canonical cliques in a Paley graph of square order $P(q^2)$ span the $\frac{-1+q}{2}$-eigenspace of the graph. This is the first of two steps to a second proof of the analogue of Erd\"os-Ko-Rado theorem for Paley graphs of square order (the first proof was given by A. Blokhuis in 1984).
\end{abstract}

\begin{keyword}
Paley graph; maximum clique; balanced characteristic function; Erd\"os-Ko-Rado theorem
\vspace{\baselineskip}
\MSC[2010] 05C25 \sep 11T99 \sep 05E30 \sep 51E15
\end{keyword}

\end{frontmatter}

\section{Introduction}
\subsection{Regular cliques in strongly regular graphs}
A clique $S$ in a regular graph is called \emph{regular} if every vertex that is not in $S$ has the same number of neighbours in $S$. The following lemma gives a characterisation of regular cliques in strongly regular graphs.

\begin{lemma}[{\cite[Proposition 1.3.2(2)]{BCN89}}]\label{HoffmanBound}
Let $\Gamma$ be a strongly regular graph with parameters $(v,k,\lambda,\mu)$ and with smallest eigenvalue $-m$.
Let $S$ be a clique in $\Gamma$. Then $|S| \le 1+\frac{k}{m}$, with equality if and only if every vertex that is not in $S$ has the same number of neighbours (namely $\frac{\mu}{m}$) in $S$.
\end{lemma}

\subsection{Strongly regular block graphs of orthogonal arrays}
We consider block graphs of orthogonal arrays (see \cite[p. 224]{GR01} and \cite[Section 5.5]{GM15}).

An orthogonal array $OA(m, n)$ is an $m \times n^2$ array with entries from an $n$-element set $X$
with the property that the columns of any $2 \times n^2$ subarray consist of all $n^2$ possible pairs.

\begin{lemma}[{\cite[Theorem 5.5.1]{GM15}}]\label{OAGraph}
If $OA(m, n)$ is an orthogonal array where $m < n+1$, then
its block graph $X_{OA(m,n)}$ is strongly regular with parameters
$$
(n^2, m(n - 1), (m - 1)(m - 2) + n - 2, m(m - 1));
$$
the eigenvalues of $X_{OA(m,n)}$ are $m(n - 1)$, $n-m$ and $-m$ with multiplicities $1$, $m(n - 1)$ and $(n - 1)(n + 1 - m)$, respectively.
\end{lemma}

Let $S_{ri}$ be the set of columns of $OA(m, n)$ that have the entry $i$ in row $r$. These sets are cliques, and since each element of the $n$-element set $X$ occurs
exactly $n$ times in each row, the size of $S_{ri}$ is $n$ for all $i$ and $r$. These cliques
are called the \emph{canonical cliques} in the block graph of the orthogonal array. However, there may be non-canonical cliques for the block graph of an orthogonal array $OA(m,n)$ in the case $n \le (m-1)^2$ (see \cite[Section 5.5]{GM15}).
Note that a clique of size $n$ in the block graph of an orthogonal array $OA(m,n)$ is necessarily regular in view of Lemma \ref{HoffmanBound} (every vertex that is not in the clique has $m-1$ neighbours in the clique).

\subsection{Finite fields of square order}
For any odd prime power $q$, the field $\mathbb{F}_{q^2}$ can be naturally viewed as the affine plane $A(2,q)$.
Let $d$ be a non-square in $\mathbb{F}_{q}^*$. The elements of the finite field of order $q^2$ can be considered as $\mathbb{F}_{q^2} =
\{x+y\alpha~|~x,y \in \mathbb{F}_q\}$, where $\alpha$ is a root of the polynomial $f(t) = t^2 - d$.
Since $\mathbb{F}_{q^2}$ is a $2$-dimensional vector space over $\mathbb{F}_q$, we can assume that the points of $A(2,q)$ are the elements of
$\mathbb{F}_{q^2}$ and a line $l$ is presented by the elements $\{x_1+y_1\alpha + c(x_2+y_2\alpha)\}$, where $x_1+y_1\alpha \in \mathbb{F}_{q^2}$,
$x_2+y_2\alpha \in \mathbb{F}_{q^2}^*$
are fixed and $c$ runs over $\mathbb{F}_q$. The element $x_2+y_2\alpha$ is called the \emph{slope} of the line $l$.
A line $l$ is called \emph{quadratic} (resp. \emph{non-quadratic}), if its slope is a square (resp. non-square) in $\mathbb{F}_{q^2}^*$.
Note that the slope is defined up to multiplication by a constant $c \in \mathbb{F}_{q}^*$.

\subsection{Paley graphs of square order}
Let $q_1$ be a prime power, $q_1 \equiv 1(4)$. Let $S_1$ be the set of squares in $\mathbb{F}_{q_1}^*$. The Paley graph $P(q_1)$ is defined on $\mathbb{F}_{q_1}^*$, where two elements $\gamma_1,\gamma_2$ are adjacent whenever $\gamma_1-\gamma_2 \in S_1$. In this paper we consider Paley graphs of square order $q_1 = q^2$, where $q$ is an odd prime power.

\begin{lemma}[{\cite[Section 10.3]{GR01}}]
Given an odd prime power $q$, the non-principal eigenvalues of the Paley graph $P(q^2)$ are $\frac{-1-q}{2}$ and $\frac{-1+q}{2}$.
\end{lemma}

\begin{lemma}[{\cite[Lemma 5.9.1]{GM15}}]\label{CliqueNumberForPaleyGraphs}
Given an odd prime power $q$, the clique number $\omega(P(q^2))$ of the Paley graph $P(q^2)$ is equal to $q$.
\end{lemma}

A characterisation of $q$-cliques in a Paley graph $P(q^2)$ was obtained in \cite{B84}; it was proved that such a clique is necessarily a quadratic line.

\subsection{Paley graphs of square order as block graphs of orthogonal arrays}

Given an odd prime power $q$, recall the construction of an orthogonal array $OA(q+1,q)$ based on the affine plane $A(2,q)$.
For the affine plane $A(2,q)$ defined by the finite field $\mathbb{F}_{q^2}$, all the $q+1$ slopes are given by the elements $\{1\} \cup \{c+\alpha~|~c \in \mathbb{F}_q\}$  (this set has exactly $\frac{q+1}{2}$ quadratic slopes; in particular, 1 is a quadratic slope). These $q+1$ slopes define an orthogonal array $OA(q+1,q)$ as follows. The rows are indexed by the slopes (by the lines through 0). The columns are indexed by the elements of $\mathbb{F}_{q^2}$. In order to define the orthogonal array, we need to assign an index with every line in the affine plane (the indexes are from $\mathbb{F}_q$). Given a line $\ell$ through 0, let us assign the index 0 with this line.

If $\ell = \mathbb{F}_q$, then an additive shift of $\ell$ is given by $\{c+j\alpha~|~c \in \mathbb{F}_q\}$ for some fixed $j \in \mathbb{F}_q$; we assign the index $j$ with this line.

If $\ell = \{c(i+\alpha)\}$ for some $i \in \mathbb{F}_q$, then an additive shift of $\ell$ is given by $\{j+c(i+\alpha)~|~c \in \mathbb{F}_q\}$ for some $j \in \mathbb{F}_q$; we assign the index $j$ with this line.

Thus, we have defined an index for every line in the affine plane (indexes for distinct parallel lines are different). Let us define an entry of the array lying in the row indexed by a line $\ell$ through 0 and the column indexed by an element $x+y\alpha$. We put the value of this entry to be the index of the line that is parallel to $\ell$ and contains the element $x+y\alpha$.
In this setting, any two rows of the array are orthogonal since every point of the plane uniquely determines a pair of intersecting lines from two fixed parallel classes.

\begin{lemma}\label{PaleyAsOAGraph}
Given an odd prime power $q$, the Paley graph $P(q^2)$ is the block graph of an orthogonal array $OA(\frac{q+1}{2},q)$.
\end{lemma}
\begin{proof}
In the above construction of the orthogonal array $OA(q+1,q)$, we take the $\frac{q+1}{2}$ rows that correspond to the quadratic slopes; these rows form an orthogonal array $OA(\frac{q+1}{2},q)$. The block graph of this array is equal to the Paley graph $P(q^2)$.
\end{proof}

\subsection{Eigenfunctions and balanced characteristic functions of cliques in Paley graphs of square order}
Given a graph $\Gamma = (V,E)$, a function $f:V\rightarrow \mathbb{R}$ is called a $\theta$-eigenfunction, if $f \not\equiv 0$ and the equality
\begin{equation}\label{eigenfunction}
\theta f(\gamma) = \sum\limits_{\delta \in N(\gamma)}f(\delta)
\end{equation}
holds for every vertex $\gamma \in V$, where $N(\gamma)$ denotes the neighbourhood of $\gamma$.

We say that a vector in $\mathbb{R}^n$ is \emph{balanced} if it is orthogonal to the all-ones vector \textbf{1}. If $v_S$ is the characteristic vector of a subset $S$ of the set $V$,
then we say that
$$
v_S - \frac{|S|}{|V|}\textbf{1}
$$
is the \emph{balanced characteristic vector (function)} of $S$ (see \cite[p. 26]{GM15}).

Let $S$ be a $q$-clique in a Paley graph $P(q^2)$. The balanced characteristic vector (function) $f_S$ is given by
$$
f_S(\gamma):=
\left\{
  \begin{array}{ll}
    \frac{q-1}{q}, & \hbox{if~$\gamma \in S$;} \\
    -\frac{1}{q}, & \hbox{otherwise.}
  \end{array}
\right.
$$
The normalised balanced function $qf_S$ is given by
$$
qf_S(\gamma):=
\left\{
  \begin{array}{ll}
    q-1, & \hbox{if~$\gamma \in S$;} \\
    -1, & \hbox{otherwise.}
  \end{array}
\right.
$$

\begin{lemma}
Given a $q$-clique $S$ in a Paley graph $P(q^2)$, its balanced characteristic vector lies in the $\frac{-1+q}{2}$-eigenspace of $P(q^2)$.
\end{lemma}
\begin{proof}
It suffices to check the condition (\ref{eigenfunction}) for the eigenvalue $\theta = \frac{-1+q}{2}$ and the normalised balanced function $qf_S$.

Take a vertex $\gamma \in S$. The left-hand side of (\ref{eigenfunction}) is equal to $\frac{(q-1)^2}{2}$.
On the other hand, $\gamma$ has $q-1$ neighbours in $S$  (with values $q-1$) and $\frac{q^2-1}{2} - (q-1)$ neighbours not in $S$ (with values -1). Thus, the right-hand side of (\ref{eigenfunction}) is equal to $(q-1)^2 - \frac{q^2-1}{2} + (q-1) = \frac{(q-1)^2}{2}$.

Take a vertex $\gamma \notin S$. The left-hand side of (\ref{eigenfunction}) is equal to $\frac{1-q}{2}$. On the other hand, since $S$ is a regular clique, $\gamma$ has $\frac{q-1}{2}$ neighbours in $S$ (with values $q-1$) and $\frac{q^2-1}{2} - \frac{q-1}{2}$ neighbours not in $S$ (with values -1). Thus, the right-hand side of (\ref{eigenfunction}) is equal to $\frac{(q-1)^2}{2} - \frac{q^2-1}{2} + \frac{q-1}{2} = \frac{1-q}{2}$.
\end{proof}

In this paper we solve Problem \ref{problem1}, which is related to a possible approach to the main theorem from \cite{B84} (see \cite[Section 16.5]{GM15}).

\begin{problem}[{\cite[Problem 16.5.1]{GM15}}]\label{problem1}
Show that $\frac{-1+q}{2}$-eigenspace of $P(q^2)$ is spanned by the balanced characteristic vectors of the canonical cliques.
\end{problem}

However, in order to have a second proof of Blokhuis's result from \cite{B84}, which is an analogue of Erd\"os-Ko-Rado theorem for Paley graphs of square order, we still need to solve Problem \ref{problem2}.

\begin{problem}[{\cite[Problem 16.5.2]{GM15}}]\label{problem2}
Prove that the only balanced characteristic vectors of sets of size $q$, in the eigenspace belonging to $\frac{-1+q}{2}$ of $P(q^2)$, are the balanced characteristic vectors of the canonical cliques.
\end{problem}

\section{Eigenfunctions of the block graphs of orthogonal arrays}
\subsection{A basis for a $(n-m)$-eigenspace of the block graph of an orthogonal array $OA(m,n)$}
A strongly regular graph $\Gamma = (V,E)$ defined as the block graph of an orthogonal array $OA(m,n)$ has $m$ partitions of the vertex set into canonical cliques of size $n$. Denote the partitions by $\Pi_1, \ldots, \Pi_m$. Let $\Pi_j = (S_{j1}, S_{j2}, \ldots, S_{jn})$ be such a partition for some $j \in \{1,\ldots, m\}$. Fix a clique from this partition, say $S_{j1}$. Given a positive integer $i \in \{2,\ldots, n\}$, define a function $f_{ji}:V \rightarrow \mathbb{R}$ as follows.
For a vertex $\gamma \in V$, put
$$
f_{ji}(\gamma):=
\left\{
  \begin{array}{ll}
    1, & \hbox{if~$\gamma \in S_{j1}$;} \\
    -1, & \hbox{if~$\gamma \in S_{ji}$;} \\
    0, & \hbox{otherwise.}
  \end{array}
\right.
$$

\begin{lemma}
For any $j \in \{1, \ldots, m\}$, $i \in \{2, \ldots, n\}$, the function $f_{ji}$ is an eigenfunction of $\Gamma$ corresponding to its largest non-principal eigenvalue $n-m$.
\end{lemma}
\begin{proof}
It suffices to check the condition (\ref{eigenfunction}) for the eigenvalue $\theta = n-m$ and the function $f_{ji}$.

Let $\gamma$ be a vertex from $S_{j1}$. The left-hand side of (\ref{eigenfunction}) is equal to $n-m$. On the other hand, $\gamma$ has $n-1$ neighbours in $S_{j1}$ (with values 1) and $m-1$ neighbours in $S_{ji}$ (with values -1). Thus, the right-hand side of (\ref{eigenfunction}) is equal to $(n-1)-(m-1) = n-m$.

In the same manner, the equality (\ref{eigenfunction}) holds for any vertex $\gamma$ from $S_{ji}$.

Let $\gamma$ be a vertex that is not from $S_{j1} \cup S_{ji}$. Then $\gamma$ has $m-1$ neighbours in $S_{j1}$ (with values 1) and $m-1$ neighbours in $S_{ji}$ (with values -1). Thus, the equality (\ref{eigenfunction}) holds.
\end{proof}

\begin{lemma}\label{MOLSBasis}
The functions $$f_{12}, f_{13}, \ldots, f_{1n}, f_{22}, f_{23}, \ldots, f_{2n}, \ldots, f_{m2}, f_{m3}, \ldots, f_{mn}$$ form a basis of the eigenspace of $\Gamma$ corresponding to the largest non-principal eigenvalue $n-m$.
\end{lemma}
\begin{proof}
It follows from Lemma \ref{OAGraph} and two facts. The first fact is that functions $f_{j_1i_1}$ and $f_{j_2i_2}$ are orthogonal if $j_1 \ne j_2$. The second fact is that, given $j \in \{1,\ldots, m\}$, the functions $$f_{j1}, f_{j2}, \ldots, f_{jn}$$
are linearly independent.
\end{proof}

\subsection{A basis for $\frac{-1+q}{2}$-eigenspace of $P(q^2)$ and its application}
The next lemma follows from Lemmas \ref{PaleyAsOAGraph} and \ref{MOLSBasis}.
\begin{lemma}\label{PaleyBasis}
Given an odd prime power $q$, the functions
$$f_{12}, f_{13}, \ldots, f_{1q}, f_{22}, f_{23}, \ldots, f_{2q}, \ldots, f_{\frac{q+1}{2}2}, f_{\frac{q+1}{2}3}, \ldots, f_{\frac{q+1}{2}q}$$ form a basis of the eigenspace of $P(q^2)$ corresponding to the largest non-principal eigenvalue $\frac{-1+q}{2}$.
\end{lemma}

Given a partition $\Pi_j = \{S_{j1}, S_{j2}, \ldots, S_{jn}\}$ for some $j \in \{1,\ldots, k\}$ and a positive integer $i \in \{1, \ldots, q\}$, define a function $g_{ji}: V(P(q^2)) \rightarrow \mathbb{R}$ by the following rule. For any $\gamma \in V(P(q^2))$,
put
$$
g_{ji}(\gamma):=
\left\{
  \begin{array}{ll}
    q-1, & \hbox{if~$\gamma \in S_{ji}$;} \\
    -1, & \hbox{otherwise.}
  \end{array}
\right.
$$
Note that the a function $g_{ji}$ is a $\frac{-1+q}{2}$-eigenfunction of $P(q^2)$, which is equal to the normalised balanced characteristic function $qf_{S_{ji}}$ of the clique $S_{ji}$.

\begin{proposition}
The functions
$$
g_{11}, g_{12}, \ldots, g_{1q}, g_{21}, g_{22}, \ldots, g_{2q}, \ldots, g_{\frac{q+1}{2}1}, g_{\frac{q+1}{2}2}, \ldots, g_{\frac{q+1}{2}q}.
$$
span the eigenspace of $P(q^2)$ corresponding to the eigenvalue $\frac{-1+q}{2}$.
\end{proposition}
\begin{proof}
It follows from Lemma \ref{PaleyBasis} and the fact that, for any $j \in \{1, \ldots, \frac{q+1}{2}\}$ and $i \in \{2, \ldots, q\}$, the equality
$$
f_{ji} = \frac{1}{q} (g_{j1} - g_{ji})
$$
holds.
\end{proof}

\begin{corollary}
The $\frac{q^2-1}{2}$ functions
$$
g_{12}, g_{13}, \ldots, g_{1q}, g_{22}, g_{23}, \ldots, g_{2q}, \ldots, g_{\frac{q+1}{2}2}, g_{\frac{q+1}{2}3}, \ldots, g_{\frac{q+1}{2}q}.
$$
form a basis of the the eigenspace of $P(q^2)$ corresponding to the eigenvalue $\frac{-1+q}{2}$.
\end{corollary}
\begin{proof}
Note that, for any $j \in \{1, \ldots, \frac{q+1}{2}\}$, the equality
$$
g_{j1} + g_{j2} + \ldots + g_{jq} = 0
$$
holds. It means that the $\frac{q^2-1}{2}$ functions
$$
g_{12}, g_{13}, \ldots, g_{1q}, g_{22}, g_{23}, \ldots, g_{2q}, \ldots, g_{\frac{q+1}{2}2}, g_{\frac{q+1}{2}3}, \ldots, g_{\frac{q+1}{2}q}.
$$
still span the $\frac{-1+q}{2}$-eigenspace of $P(q^2)$ and thus form a basis since the $\frac{-1+q}{2}$-eigenspace has dimension $\frac{q^2-1}{2}$.
\end{proof}

\section*{Acknowledgments}
Sergey Goryainov is supported by RFBR according to the research project 20-51-53023. Sergey Goryainov is grateful to Alexander Gavrilyuk and Alexandr Valyuzhenich for valuable discussions concerning the paper. Huiqiu Lin is supported by National Natural Science Foundation of China (Nos. 11771141 and 12011530064)

\end{document}